\documentclass[10pt]{article}
\usepackage{amsmath}
\usepackage{amsfonts}
\usepackage{amsthm}
\usepackage{amssymb}
\usepackage{enumerate}
\newtheorem{thm}{Theorem}[section]
\newtheorem*{Thm}{Theorem}

\newtheorem{lemma}[thm]{Lemma}
\newtheorem{prop}[thm]{Proposition}

\newtheorem{defin}{Definition}[section]

\DeclareMathOperator{\Hom}{Hom}
\DeclareMathOperator{\im}{im}
\title{Fitting quotients of finitely presented abelian-by-nilpotent groups}
\author{J. R. J. Groves \and Ralph Strebel}
\date{}
\begin{document}
\maketitle
\begin{abstract}We show that every finitely generated nilpotent group of class 2 occurs as the quotient of a finitely presented abelian-by-nilpotent group by its largest nilpotent normal subgroup.\end{abstract}
\pagestyle{myheadings}
\markboth{Fitting quotients of finitely presented abelian-by-nilpotent groups}%
{Fitting quotients of finitely presented abelian-by-nilpotent groups}

\emph{Names and adresses}

J. R. J. Groves, Department of Mathematics and Statistics, University of Melbourne, Parkville, Victoria 3010, Australia, jrjg@unimelb.edu.au

Ralph Strebel, Department of Mathematics, University of Fribourg, 1700 Fribourg, Switzerland,
ralph.strebel@unifr.ch
\smallskip 

\emph{Suggested running title} \\Fitting quotients of finitely presented abelian-by-nilpotent groups

\begin{section}{Introduction} 

Brookes showed in \cite{BRjgt} that a finitely presented abelian-by-polycyclic group is virtually an extension of a nilpotent group by a nilpotent group of class at most two. In this paper we show the following.
 
\begin{Thm} Let $Q$ be a finitely generated nilpotent group of class 2. Then there exists a finitely presented group $G$ with an abelian Fitting subgroup $A$ so that the quotient  $G/A$ is isomorphic to $Q$.
\end{Thm}
 
 If $Q$ were a finitely generated abelian group then this result would be a very straightforward one.  But the arguments needed here cannot be a simple extension of those straightforward techniques. 
 In \cite{jrjgcjbb5}, Brookes and  the first author have  proved a theorem which implies, amongst other things, that if a finitely presented abelian-by-nilpotent group is subdirectly irreducible, then the quotient by the Fitting subgroup is  virtually a central product of generalized Heisenberg groups. 
 (By a {\it generalized Heisenberg} group we mean a group which is torsion-free nilpotent of class at most 2 and has cyclic centre; this includes the infinite cyclic group.) Thus our arguments here will require groups which are subdirect  products in a non-trivial way.
 
 The basic idea of the proof is straightforward. We express the nilpotent group $Q$ as a subdirect product of subdirectly irreducible groups $Q_i$ and for each of those groups, we produce a module $B_i$ so that the extension of this module by $Q_i$ is finitely presented. The group $G$ is then a split extension of $Q$ by the direct sum $A$ of the modules $B_i$. 
The heart of the paper is in showing how to choose these modules so that the group $G$ is finitely presented. 
We do this with the aid of a theory, developed by the second author, 
of $\Sigma$ invariants for modules over nilpotent groups.

Some of the ideas in Sections 2 and 3 of the paper arose initially in conversation with Chris Brookes. 
We take this opportunity to thank him for his contribution. 
We are also grateful to the referee for his or her careful reading of the paper.
\end{section}
\begin{section}{Subdirect products of nilpotent groups}
The following result is presumably well-known. 
\begin{prop} \label{subdir} Let $H$ be a finitely generated nilpotent group of class at most 2.
\begin{enumerate}[(i)]
\item $H$ is a subdirect product of a finite group and a finite number of torsion-free groups 
having infinite cyclic centre.
\item If $H$ has infinite cyclic centre then, for some $k\ge 0$, it has a presentation with generators $\{x_1, \dots x_k, y_1, \dots, y_k, z\}$ and with relations
\begin{equation}\begin{split} \label {hk}[x_i, x_j]&=[y_i,y_j]=1 \text{ for all }1\le i < j\le k,\\
 [x_i,y_j]&=z^{\delta_{ij}m_i}\text{ for all }1\le i, j\le k,\\
 [x_i, z]&= [y_i,z] = 1 \text{ for all } 1 \leq i \leq k 
 \end{split}
 \end{equation}
 for suitable non-zero integers $m_i$.
\end{enumerate}
\end{prop}
A group which is presented as above will be called a {\em generalized Heisenberg group of rank $k$} and the set $\{x_1, y_1, \dots, x_k,y_k\}$ will be called a {\em symplectic basis} for $H$. Note that for $k=0$, such groups are infinite cyclic.
\begin{proof} (i) 
Because $H$ is finitely generated nilpotent it has a finite torsion subgroup $T$ and it is residually finite. Thus there is a normal subgroup $N$ of finite index which meets $T$ trivially 
and so $H$ is the subdirect product of the torsion-free group $Q =H/T$ and $H/N$. 
The centre $Z$ of $Q$ is torsion-free, hence free abelian, say of rank $r$.
There exist therefore $r$ direct summands $N_i$ of $Z$ having rank $r-1$ and intersecting in $\{1\}$.
Put $Q_i = Q/N_i$. Then $Q$ is a subdirect product of the torsion-free quotient groups $Q_i$.
Since Êthe Hirsch length of each $Q_i$ is smaller than that of $Q$, 
we may assume inductively
that each quotient $Q_i$ is a subdirect product of subdirectly irreducible torsion-free quotient groups $Q_{ij}$.
So $Q$ is the subdirect product of all the $Q_{ij}$ 
whence $H$ is the subdirect product of the subdirectly irreducible torsion-free quotient groups $Q_{ij}$ 
and the finite quotient group $H/N$.
Notice that the previous argument shows 
that the centre of a subdirectly irreducible torsion-free quotient is infinite cyclic.
\smallskip

 (ii) Suppose that $H$ has cyclic centre $Z$ generated by $z$. Then there is a symplectic form on $H/Z$ given by, for $x,y\in H$,
$$(xZ, yZ)=m \text{ exactly when } [x,y]=z^m$$
It is a standard result (see, for example, \cite[$\S$ 5]{Bou07}) 
that $H/Z$ then has a symplectic basis as $\mathbb Z$-module. Lifting this symplectic basis to $H$ we arrive at exactly the type of presentation described in the proposition. 
\end{proof}
Observe that  there is room for choice of symplectic basis in the proof of the second part of the proposition. We will make use of this later. 
\end{section}
\begin{section}{Modules for generalized Heisenberg groups}\label{modules}
A similar construction to the one described here appears in the thesis of Dugal Ure \cite{UreMscThesis}. 
We have included details here both for completeness and because we know of no easily accessible source.
Suppose that $H$ is a generalized Heisenberg group of rank $k\geq 0$. Choose a generating set  $\{x_1, \dots x_k, y_1, \dots, y_k, z\}$ for $H$ and suppose that $H$ then has presentation (\ref{hk}). We shall describe the construction of a module $A$ for $H$ so that the split extension $G$ of $A$ by $H$ is finitely presented and $A$ is the Fitting subgroup of $G$.

We shall, in fact, use HNN extensions to construct inductively groups $G_0$, \ldots, $G_k$ 
so that $G_k$ is the required group $G$ and we will recover $A=A_k$ as a normal subgroup of $G_k$.

Let $L$ be the free abelian group with basis $\{ x_1, \dots, x_k\}$; 
let $A_{-1}$ be the free $\mathbb ZL$-module with basis $\mathcal B=\{a_1, a_2\}$ and let $G_{-1}$ be the split extension of $A_{-1}$ by $L$.  Then $G_{-1}$ has  a presentation with generators $\{x_1, \dots, x_k, a_1, a_2\}$ and relations
\begin{equation}  \label{pres1}  [x_i, x_j]=1\  (1\le i <  j\le k)\qquad [a',a^w]=1\ (a,a' \in \mathcal B, w\in L)  \end{equation}
Using this presentation, it is easy to see that the assignments $x_i\mapsto x_i,  a_1\mapsto a_1a_2, a_2 \mapsto a_1a_2^2$ extend to an endomorphism of $G_{-1}$. It is equally simple to write down an inverse for this endomorphism and so it is an automorphism $\phi_0$ of $G_{-1}$. Let $G_{0}$ be the HNN extension (in this case, split extension) of $G_{-1}$ corresponding to $\phi_0$ and with stable letter $z$. Thus $G_0$  has an extra generator $z$ and extra relations
\begin{equation}  \label{pres2}[x_i,z]=1\  (1\le i\le k),\qquad  a_1^z=a_1a_2,\ a_2^z=a_1a_2^2. \end{equation}

Suppose now that we have constructed, by a series of HNN extensions, a group $G_{r-1}$ with the generators of $G_0$ together with new generators $y_1, \dots, y_{r-1}$ and the relations of $G_0$ together with new relations 
\begin{equation}  \label{pres3} [x_i, y_j]= z^{\delta_{ij}m_i},\  [y_j,z]=1, \ \ [y_j,y_l]=1,\ \qquad  a^{y_j}=aa^{x_j} \end{equation}
where $1\le i\le k$,  $1\le j,  l\le r-1$ and $a\in \mathcal B$.

We now construct $G_r$. 
Consider the assignment $\phi_r$ 
that fixes $z$, $y_1$, \ldots, $y_{r-1}$ and each  $x_i \in L$ except $x_r$,
and satisfies  $x_r \mapsto x_rz^{m_r}$ and $a\mapsto aa^{x_r}$ for $a\in \mathcal B$. 
It is easily verified that $\phi_r$ preserves the relations of $G_{r-1}$ as given in (\ref{pres1}), (\ref{pres2}) and (\ref{pres3}). Thus $\phi_r$ induces an endomorphism of $G_{r-1}$. 
To see that $\phi_r$ is injective, 
observe firstly that $G_{r-1}$ is
obtained by a series of ascending HNN-extensions with stable letter $z$ or some $y_i$, 
each of which is fixed by $\phi_r$. 
Thus it will suffice to establish that the restriction of $\phi_r$ to $G_{-1}$ is injective. 
Now $\phi_r$  induces an automorphism of the group $L$, 
hence a ring automorphism of the group ring $\mathbb{Z}L$;
on the other hand,
since $\mathbb{Z}L$ is an integral domain, 
the assignments $a_1 \mapsto a_1 a_2$ and $a_2 \mapsto a_1 a_2^2$ 
induce an injective endomorphism of the $L$-module $A_{-1}$. 
It follows that $\phi_r \colon G_{-1} \to G_{r-1}$ is injective.

Thus we can again construct an HNN-extension $G_r$ corresponding to $\phi_r$ and with stable letter $y_r$. 
This group $G_r$ has a presentation with generators 
$\{x_1,\dots, x_k, y_1, \dots, y_r, z, a_1, a_2\}$ and relations  (\ref{pres1}) and (\ref{pres2}) together with (\ref{pres3}), but with $r-1$ replaced by $r$.
Repeating this process, we eventually arrive at the group $G_k$ with the generating set $\{x_1,\dots, x_k, y_1, \dots, y_k, z, a_1, a_2\}$ and relations given by (\ref{pres1}) and (\ref{pres2}) together with (\ref{pres3}), 
but with $r-1$ replaced by $k$. 
Let $A=A_k$ denote the normal closure of $\mathcal B$ in $G=G_k$. 
The group $A$ contains the subgroup $A_{-1}$ which is abelian;
as $A$ is obtained from $A_{-1}$ by a sequence of localizations, $A$ itself is abelian.
From the presentation, it is easy to verify that $G/A\cong H$. Also, if there were a nilpotent normal subgroup of $G$ larger than $A$, then some non-trivial element of the centre of $H$ would have to act nilpotently on $A$. But it is clear from the definition that no non-trivial power of $z$ acts nilpotently on $A$. 
Thus $A$ is the Fitting subgroup of $G$. 
The final step, the justification that $G$ is finitely presented, will be deferred to the next section.
\end{section}
\begin{section}{Geometric invariants \\for modules over nil\-po\-tent groups}
\subsection{Definition of the invariants}
We shall give only a very brief description of the invariants here. 
The subject is treated in detail and with proofs in a paper of the second author \cite{RSfpan}.

Let $Q$ be a finitely generated group. By a {\em character} of $Q$ we mean a homomorphism $Q \rightarrow \mathbb R$; that is, an element of the dual $Q^*$ of $Q$. 
Then $Q^*$ is a finite dimensional real vector space.

The usual version of the invariant we want to use is obtained by identifying characters which differ by multiplication by a positive real 
and so replacing the vector space $Q^*$ by a sphere. 
Most of the arguments in this paper will require a vector space, however, 
and so we will translate results and terminology from \cite{RSfpan} accordingly.

Let $A$ denote a finitely generated $\mathbb ZQ$-module. 
The subset $\Delta(Q; A)$ of $Q^*$ consists of the equivalence classes of the complement
$S(G) \smallsetminus \Sigma^0(Q;A)$ of the set $\Sigma^0(Q;A)$ defined in \cite{RSfpan}. 
More directly, if $\chi\in Q^*$ 
then $\chi \in \Delta(Q;A)$ if and only if,
either $A$  is not finitely generated as $\mathbb ZQ_{\chi}$-module, or $\chi = 0$.
Here $Q_{\chi}$ is the submonoid of $Q$ consisting of the elements with non-negative $\chi$-value.

In \cite{RSfpan}, the term {\em tame} is defined with respect to an arbitrary central series of a group. We shall use it, however, only for the very simple case of the lower central series of a group of class 2. 
We therefore define it only in this special case. 

\begin{defin}\label{tame}
Suppose that $Q$ is a nilpotent group of class 2 and $A$ is a $\mathbb ZQ$-module with finite generating set $\mathcal A$. We say that $A$ is {\em tame} (for the lower central series) if both $\Delta(Q; A)$ and $\Delta(Q'; \mathcal  A\cdot \mathbb ZQ')$ contain no lines (equivalently, no diametrically opposite points).
\end{defin}

The key theorem we shall use is the following, which is a special case of a much more general result appearing as Theorem 4.1  of \cite{RSfpan};
\begin{thm} 
\label{RSthm}With the notation above, if $A$ is tame then the split extension of $A$ by $Q$ is finitely presented. \end{thm}

\subsection{Calculating the invariant of the examples}\label{invex}
We now turn to showing that the group $G$, constructed in the previous section, is finitely presented. Clearly, by Theorem \ref{RSthm}, it suffices to show that the module $A$ constructed there for the subdirectly irreducible group $H$ is tame. We begin with the trivial cases. If $H$ is finite or infinite cyclic then $A$ is finitely generated as abelian group and so $\Delta(H; A)$ consists only of  $\{0\}$. (In fact, in this case, the extension of $A$ by $H$ is polycyclic.)

We shall use the following criterion for calculating $\Delta$ 
which has been adapted from the more general situation of  \cite{RSfpan}. 
\begin{quote} 
\emph{Let $A$ be a cyclic module with generator $a$. Then $\chi \notin \Delta(Q; A)$ if and only if there exists $\alpha \in \mathbb ZQ$ so that $a=a\alpha$ and $\chi(q)>0$ for each $q$ in the support of $\alpha$.}
\end{quote}

In the case when $H$ is a generalized Heisenberg group of rank $k>0$, 
then we observe that either element $a$ of $\{a_1, a_2\}$ is a generator for $A$ as $\mathbb ZQ$-module 
and that the defining relations of $H$  imply that $a$ is annihilated by each of the elements
\begin{equation}
\label{eq:Annihilating-elements-of-a1}
1 + x_1-y_1,\quad 1 + x_2-y_2, \ldots,   1 + x_k-y_k.
\end{equation}
(Note that we regard $A$ as {\em right} $\mathbb ZQ$-module.)
These annihilating elements and the previously quoted criterion
allow us to find an upper bound for $\Delta(H; A)$ as follows.

Fix an index $i \in \{1, \ldots, k\}$ and let $\chi \colon H \to \mathbb R$ be a character
that assumes its minimum only once on the support $\{1, x_i, y_i \}$ of $1 + x_i - y_i$.
If the minimum occurs at 1 then we rewrite the equation $a \cdot (1 + x_i - y_i) = 0$ in the form 
$a = a (-x_i + y_i)$  and deduce from the criterion 
that $\chi \notin \Delta(H;A)$ if $\chi(x_i) > 0$ and $\chi(y_i) > 0$.
If the minimum is taken on $x_i$ we use that $a$ is annihilated by 
$(1 + x_i - y_i) \cdot x_i^{-1} = 1 + x_i^{-1} - y_i x_i^{-1} $ 
and infer that $\chi \notin \Delta(H;A)$ if $\chi(x_i^{-1}) > 0$ and $\chi(y_i x_i^{-1}) > 0$;
if it is taken on $y_i$ we argue similarly.
We conclude that $\chi$ can only belong to $\Delta(H;A)$ if it assumes its minimum at least twice 
on the support of each of the elements $1 + x_i - y_i$.

It follows that  $\Delta(H;A)$ is contained in the union of $3^k$ subsets.
To describe these sets concisely,
we introduce characters $\chi_i$, $\psi_i$, $\phi_i\in H^*$ defined by 
$$\chi_i(x_i)=1,\quad \psi_i(y_i)=1, \quad \phi_i(x_i)=\phi_i(y_i)=-1,$$ 
all values on other generators $x_j$ or $y_j$ being zero. 
The subsets are then as described by the first statement of the next proposition;
the second statement is a simplified version of the first which will be used later.

\begin{prop}\label{deldesc}
 Let $H$ and $A$ be as defined in Section 2 and let $\chi_i$, $\psi_i$ and $\phi_i = -\chi_i - \psi_i$  in $H^*$ be as before. Then
\begin{enumerate}[(i)]
\item $\Delta(H; A)$ lies in the union of the $3^k$ sets, each defined as the convex cone generated by a set of the form $\{\lambda_1, \dots, \lambda_k\}$ where each $\lambda_i$ is an element of $\{\chi_i, \psi_i, \phi_i\}$.
\item  $\Delta(H; A)$ lies in the union of the $3^k$ subspaces each spanned by a set of the form $\{\lambda_1, \dots, \lambda_k\}$ where each $\lambda_i$ is an element of $\{\chi_i, \psi_i, \phi_i\}$. 
\end{enumerate}
\end{prop}

\subsection{The invariant for the module in the theorem}
\label{ssec:Modules-theorem}
The $Q$-module $A$ occurring in the statement of our theorem will be obtained 
by combining three constructions.
To begin with,
the given group $Q$ is expressed as a subdirect product of a finite group $Q_1$
and generalized Heisenberg groups $Q_2$, \ldots, $Q_m$.
For each factor $Q_i$ one constructs a tame $Q_i$-module $B_i$, pulls its action back to $Q$
and then chooses  $A$ to be the direct sum $B_1 \oplus \cdots \oplus B_m$.
To calculate $\Delta(Q,A)$ one then uses the direct sum formula 
\begin{equation}
\label{eq:Direct-sum}
\Delta(Q;A_1 \oplus A_2) = \Delta(Q;A_1) \cup \Delta(Q;A_2);
\end{equation}
it is valid for every couple $A_1$, $A_2$ of finitely generated $Q$-modules (cf.\, \cite[Lemma 2.2]{RSfpan}).
This leaves us with the calculation of the invariants $\Delta(Q;B_i)$.
Each projection $\pi_i \colon Q \twoheadrightarrow Q_i$ induces an injective linear map 
$\pi_i^* \colon Q_i^* \rightarrowtail Q^*$ and 
\begin{equation}
\label{eq:Pullback}
\Delta(Q;B_i) = \pi_i^*\left(\Delta(Q_i;B_i)\right);
\end{equation}
(cf.\, \cite[Lemma 2.5]{RSfpan}).
The calculation of $\Delta(Q;A)$ can thus be reduced 
to that of the subsets $\Delta(Q_1;B_1)$, \ldots, $\Delta(Q_m; B_m)$.

Formula \eqref{eq:Pullback} implies that $\Delta(Q; B_i)$ contains no line if $\Delta(Q_i; B_i)$ has this property;
but such a conclusion cannot be drawn from formula \eqref{eq:Direct-sum}: 
if $\Delta(Q;A_1)$ and $\Delta(Q;A_2)$ both contain a half-line their union can be a line.

If $Q_i$ is finite or infinite cyclic this problem does not arise, 
since $B_i$ can then be chosen so that $\Delta(Q_i;B_i)$ is reduced to the origin.
If, however, $Q_i$ is a Heisenberg group of rank $k > 0$ 
no easy solutions seems available.
The solutions proposed in Section \ref{Proof-theorem}
will involve the choice of a symplectic basis of a normal subgroup $P$ of $Q_i$ of finite index
and the module $B_i$ will be an induced module $B \otimes_P \mathbb{Z}Q_i$ with $B$ a $P$-module of the kind constructed in Section \ref{modules}. 
The set $\Delta(Q_i; B \otimes_P \mathbb{Z}Q_i)$ can then be computed with the help of the formula
\begin{equation}
\label{eq:Induced}
\iota^*\left(\Delta(Q_i; B \otimes_P \mathbb{Z}Q_i) \right)= \Delta(P;B) \cap \im \iota^*.
\end{equation}
In the above $\iota \colon P \leq Q_i$ denotes the inclusion 
and $\iota^*$ is the induced linear map $Q_i^* \rightarrowtail  P^*$
(see \cite[Lemma 2.4]{RSfpan}). 
\end{section}
\begin{section}{Symplectic spaces}
\label{Symplectic-spaces}
Let $\mathbb F $ be a subfield of $\mathbb R$ 
and  $(V, \beta)$ a finite dimensional vector space over $\mathbb F$
equipped with a non-degenerate skew-symmetric bilinear form $\beta \colon V \times V \to \mathbb F$.
In the sequel, $(V, \beta)$ will be referred to as a \emph{symplectic space}.
We shall make use of the following basic properties of $(V, \beta)$:
\begin{enumerate}
\item[a)] $V$ is even dimensional, say $\dim_{\mathbb F}V = 2k$;
\item[b)] $V$ admits a  basis $\{e_1, f_1, \ldots, e_k, f_k\}$, 
called \emph{symplectic}, such that  
\[
\beta(e_i, _j) = \beta (f_i, f_j) = 0 \text{ and }\beta(e_i, f_j) = \delta_{i,j} \text{ for all } i, j;
\]
\item[c)] for every subspace $U$ the \emph{annihilator} 
\[
U^\perp = \{ v \in V \mid \beta(v,u) = 0 \text{ for all } u \in U \} 
\]
has dimension $2k - \dim U$ and $(U^\perp)^\perp = U$
\end{enumerate}
(see, for instance, \cite{Vin03}, Corollary 5.60, Theorem 5.59 and  Proposition 5.43).

A subspace $U$ of $V$ with $U \subseteq U^\perp$ is called \emph{isotropic};
 it is said to be \emph{Lagrangian} if it is isotropic and not contained in a larger isotropic subspace.
It is clear that every isotropic subspace is contained in a Lagrangian subspace; 
moreover, the formula $\dim U^\perp = \dim V - \dim U$ (see c) above) implies 
that every Lagrangian subspace is of dimension $\tfrac{1}{2} \dim V$.
If $\{e_1, f_1, \ldots, e_k, f_k\}$ is a symplectic basis of the symplectic space $(V,\beta)$
the subspace $U$ spanned by $e_1$, \ldots, $e_k$ is Lagrangian;
conversely, 
every Lagrangian subspace is of this form (easy induction on $k$). 

As before,
let $(V, \beta)$ be a symplectic space of dimension $2k$ over a subfield $\mathbb F $ of $\mathbb R$
and let $\Omega$ be a finite collection of subspaces of $V$ each having dimension at most $k$.
In the arguments of this section, we will use without further comment the well-known fact 
that a finite dimensional vector space over an infinite field 
cannot be the union of a finite collection of proper subspaces.
 \begin{lemma}
There is a Lagrangian subspace $L$ of $V$ so that $L\cap W=\{0\}$ for each $W\in \Omega$. \end{lemma}
\begin{proof} 
We shall prove the result by induction on $k$; 
if $k=1$ it is immediate.
We can clearly suppose that each $W\in \Omega$ is non-zero; 
as $\beta$ is non-degenerate, 
this implies that each annihilator $W^{\perp}$ is a proper subspace.
Thus we can find $u\in V$ so that, for each $W\in \Omega$,
the vector  $u$ belongs to neither $W$ nor $W^{\perp}$. 
Let $U$ denote the subspace spanned by $u$ and let $U^{\perp}$ be its annihilator. 
Then $U$ has dimension 1 and $U^{\perp}$ has dimension $2k-1$. 
Further, as $u\notin W^{\perp}$ it follows that each $W\in \Omega$ is not a subspace of $U^{\perp}$. 
Thus $W\cap U^{\perp}$ will have dimension $\dim(W)-1$ 
and its image $\hat W$ in $U^{\perp}/U$ will also have dimension $\dim(W)-1$.

The symplectic form $\beta$ induces a skew-symmetric form on the $2(k-1)$-dimensional space $U^{\perp}/U$;
since $(U^\perp)^\perp = U$ by property c),
the induced form is non-degenerate.
Moreover,  
the set $\hat \Omega=\{\hat W : W \in \Omega\}$ is a finite set of subspaces of dimension at most $k-1$. 
Thus the inductive hypothesis provides a Lagrangian subspace $\hat L$ of $U^{\perp}/U$ 
so that $\hat L\cap \hat W=\{0\}$. Let $L$ be the lift of $\hat L$ in $U^{\perp}$. 
If $W\in \Omega$, then $L\cap W$ must lie in $U$ and so must be zero as $u\notin W$.
\end{proof}

The next lemma is the technical heart of this section. Let $L$ be a Lagrangian subspace of $V$ and suppose that each element of $\Omega$   has trivial intersection with $L$. Choose a basis $\{e_1, \dots. e_k\}$ for $L$ and extend it to a symplectic basis $\{e_1,f_1,\dots, e_k,f_k\}$ for $V$.

Given $\mu=(\mu_1, \dots,\mu_k)  \in [0,1]^k$, let $K_{\mu}$ be the span of the vectors 
$$\{ (1-\mu_1)e_1+\mu_1f_1, \dots , (1-\mu_k)e_k+\mu_kf_k\}$$
Thus if $\mu=(0, \dots,0)$ then $K_{\mu}=L$ and if  $\mu=(1, \dots,1)$ then $K_{\mu}$ is the span of $\{f_1, \dots,f_k\}$.

\begin{lemma}
\label{lem:Properties-K-sub-mu}
The family $K_\mu$ of subspaces has the following properties:
\begin{enumerate}[(i)]
\item for every $\mu \in [0,1]^k$ the subspace $K_\mu$ is  Lagrangian;
\item for every $\mu \in (0,1]^k$ the subspace $K_\mu$ is transversal to $L$;
\item if every subspace in $\Omega$ has dimension $k$ then  for every $\mu$ sufficiently close to $(0, \ldots, 0)$ 
the subspace $K_\mu$ is transversal to every subspace in $\Omega$.
\end{enumerate}
\end{lemma}

\begin{proof}
(i) and (ii).
For every index $i$, the vectors $e_i$ and $f_i$ span a symplectic plane; 
these planes are orthogonal to each other.
It follows that each subspace $K_\mu$ is isotropic and of dimension $k$, hence Lagrangian.
Moreover, 
if $\mu \in (0,1]^k$ the subspace $L + K_\mu$ contains a basis of $V$ 
whence the Lagrangian subspaces $L$ and $K_\mu$ must be transversal.

(iii)
Let $W$ be a subspace in $\Omega$ and let $w_1$, \ldots, $w_k$ a basis of $W$.
Consider the function
\[
g \colon [0,1]^k \to \mathbb R, 
\quad 
\mu \mapsto \det ((1-\mu_1)e_1 + \mu_1 f_1, \ldots, (1-\mu_k)e_k + \mu_k f_k , w_1, \ldots, w_k).
\]
This function is continuous and is non-zero at $(0, 0, \ldots, 0)$ (since $W$ is a complement of $L$).
Hence there exists an interval $[0,\varepsilon_W]$ such that $g$ does not vanish on $[0,\varepsilon_W]^k$.
Set $\varepsilon = \min \{\varepsilon_W \mid W \in \Omega \}$.
For every $\mu$ with $0 < \mu_i \leq \varepsilon$ for $1 \leq i \leq k$,
the subspace $K_\mu$ is then a Lagrangian subspace 
that is transversal to $L$ and to every subspace in $\Omega$.
\end{proof}

\begin{defin} \label{assoc} Let $\mathcal C=\{\chi_1, \psi_1, \dots \chi_k,\psi_k\}$ be a symplectic basis 
for $V$. 
Then the {\em associated subspaces} corresponding to the basis $\mathcal C$ are all subspaces of the form  
$\langle\lambda_1, \dots, \lambda_k\rangle$ where each $\lambda_i$ is an element of 
$\{\chi_i, \psi_i, \chi_i+\psi_i\}$.
\end{defin}
Note that, in this language,  
the set of subspaces described in (ii) of Proposition \ref{deldesc} is the set of subspaces associated to the basis $ \{\chi_1, \psi_1, \dots \chi_k,\psi_k\}$ of $H^*$ 
which is dual to the basis  $\{x_1, y_1, \dots, x_k, y_k\}$ of $H$. (Here $\mathbb F  =\mathbb R$). 
\begin{prop}
\label{thm:Existence-simultaneous-complement}
There exists a symplectic basis $e_1$, $\hat{f}_1$, \ldots, $e_k$, $\hat{f}_k$ 
such that every subspace associated to it has trivial intersection with each subspace in $\Omega$.
\end{prop}

\begin{proof} Observe that it is sufficient to prove the theorem in the case that each subspace in $\Omega$ is replaced by a subspace of dimension $k$ which contains it. 
Let $L$, the symplectic basis $\{e_1, f_1, \dots, e_k, f_k\}$  and the family $K_\mu$ be as described in Lemma \ref{lem:Properties-K-sub-mu}.
Fix a positive integer $p$ 
so that $K_\mu$ is transversal to every subspace in $\Omega$ for each $\mu \in [0, \tfrac{1}{p}]^k$.
Set $\hat{f_i}=p(e_i+\frac 1p f_i)$. Then the set $\{e_1,\hat{f_1}, \dots, e_k, \hat{f_k}\}$ is also a symplectic basis for $V$. 

Further, any subspace associated to this basis is spanned by  a set of the form $\{u_1, \dots, u_k\}$ where $u_i$ is one of the vectors 
$$e_i, \hat{f_i}=p(e_i+\frac 1pf_i) \text{ or } e_i+\hat{f_i}=(p+1)\left(e_i+\frac 1{(p+1)}f_i\right). $$
Thus each associated subspace is of the form $K_{\mu}$ with $\mu\in  [0, \tfrac{1}{p}]^k$ and so meets every subspace $W \in \Omega$ trivially.
\end{proof}
\end{section}
\begin{section}{Proof of the Theorem}
\label{Proof-theorem}
Suppose that $Q$ is a finitely generated nilpotent group of class 2. 
Write $Q$ as a subdirect product with factors $Q_1$, \dots, $Q_m$ 
where $Q_1$  is finite and the remaining factors are torsion-free groups with infinite cyclic centre
and so generalized Heisenberg groups (see Proposition \ref{subdir}).
We then have epimorphisms $\pi_i\colon Q \rightarrow Q_i$ and $\bigcap_{i=1}^m \ker(\pi_i)$ is trivial.

We shall construct the group $G$ by forming a module $A$ for $Q$ and then defining $G$ as a split extension of $A$ by $Q$. We will take, for each $i$,  a suitable module $B_i$ for $Q_i$ and then regard $B_i$ as a $Q$-module via $\pi_i$. The module $A$ will then be the the direct sum of the $B_i$. We must exercise considerable care, however, to ensure that the resulting module $A$ is tame and hence that $G$ is finitely presented.

The factor $Q_1$ is finite;  
order the remaining factors so that for $i \leq j$
the rank of $Q_i$, as generalized Heisenberg group, is less than or equal to the rank of $Q_j$.  
For the finite factor $Q_1$ the module $B_1$ will be the group ring of $Q_1$; 
note that $\Delta(Q_1; B_1)$ is reduced to 0. 
In all other cases, 
the module will be a modified version of the modules $A$ constructed for generalized Heisenberg groups 
in Section \ref{modules}. 
If $Q_i$ is cyclic, we can use such a module for $B_i$ without further adjustment. In this case, the module $B_i$ is finitely generated as abelian group and the corresponding $\Delta(Q_i; B_i)$ is again reduced to the origin. 
For the other $B_i$ we will need to take more care to ensure that the resulting direct sum of the modules $B_i$ is tame.

Suppose that, for some $l$ with $1\le l\le m$, 
we have constructed modules $B_1, \dots, B_{l-1}$ for $Q$ 
so that the set 
\[
\Delta(Q; B_1\oplus\dots\oplus  B_{l-1})=\Delta(Q; B_1)\cup\dots\cup  \Delta(Q; B_{l-1})
\]
contains no lines
and is contained in a finite collection $\Omega$ of subspaces of $Q^{*}$, 
each of dimension no more than the rank $k$ of $H = Q_l$
and spanned by at most $k$ characters with images in $\mathbb{Q}$. 
In order to describe our construction of the module $B_l$, 
we need to consider symplectic forms on $H$ 
and on $\Hom(H,\mathbb Q)$; 
we shall denote $\Hom(H,\mathbb Q)$ by $H^{\#}$. 
The projection $\pi = \pi_l \colon  Q\twoheadrightarrow H = Q_l$  yields an embedding $H^* \rightarrowtail Q^*$ 
and the inclusion of the rational numbers in the reals yields a further embedding $H^{\#} \rightarrowtail H^*$. 
We shall treat these embeddings as inclusions.

We have observed in the proof of Proposition \ref{subdir} that the quotient $H/Z$ by the centre
has a natural symplectic form given by commutation. 
This yields an embedding $\rho\colon H/Z \rightarrow H^{\#}$ defined by
$$\rho(h_1Z)[h_2] = (h_1Z, h_2Z)$$
and the image of $\rho$ in $H^{\#}$ is a lattice in $H^{\#}$ of full rank; 
that is, it contains a basis of $H^{\#}$. It follows that $H^{\#}$ inherits a symplectic form $\beta$ satisfying 
\[
\beta(\rho(h_1Z), \rho(h_2Z))=(h_1Z,h_2Z) =s \text{ exactly when } [h_1,h_2]=z^s
\]
where $z$ is the chosen generator of $Z$.

Let $\hat{\Omega}=\{ W\cap H^{\#}: W \in \Omega\}$. 
Then $\hat\Omega$ is a finite collection of subspaces of $H^{\#}$ each of which has dimension no more than $k$
(recall that each $W \in \Omega$ is spanned by at most $k$ elements of $H^{\#}$).
Thus we can apply Proposition \ref{thm:Existence-simultaneous-complement} 
(with $\mathbb F = \mathbb Q$) to find a symplectic basis $\mathcal C$ for $V = H^{\#}$ 
so that each associated subspace of this symplectic basis avoids each element of $\hat\Omega$ 
and hence avoids each element of $\Omega$.

We shall now show how to construct the module $B_l$. 
As $\rho(H/Z)$ is a full lattice in $H^{\#}$, 
there is a positive  integer $j$ so that $j\mathcal C \subseteq \rho(H/Z)$. 
Thus there is a subset $\mathcal D \subseteq H$ so that $\rho$ maps $\{dZ \colon d\in \mathcal D\}$ 
bijectively onto $j\mathcal C$.
It follows easily that $\mathcal D$ is a symplectic basis for a normal subgroup $P$ of finite index in $ Q_l= H$. 
We then proceed with the construction of Section \ref{modules} to produce a module $B$ for $P$ 
so that the set  $\Delta(B)\subseteq P^*=Q_l^*$ 
will lie in the union of the subspaces associated to the symplectic basis $\mathcal D$ of the Heisenberg group $P$; 
because $\mathcal C =j^{-1} \rho(\mathcal D)$ these associated subspaces coincide with the subspaces associated to the symplectic basis $\mathcal C$ of the vector space $V = H^{\#}$.
We then induce $B$ from $P$ to $Q_l$ to obtain the module $B_l$.
It follows from formula \eqref{eq:Induced} and the fact 
that $\iota \colon P \leq H$ induces an isomorphism $\im^* \colon H^* \to P^*$
that $\Delta(Q_l;B_l)$ lies in the union of the subspaces 
associated to the symplectic basis $\mathcal C$.

We have now constructed a module $B_l$ for $Q_l$ so that $\Delta(Q_l ; B_l)$ contains no lines.
We pull $B_l$ back to $Q$;
formula \eqref{eq:Pullback} then shows that $\Delta(Q;B_l)$ contains no lines either
and that it lies in a finite union of subspaces each of which intersects the subspaces in $\Omega$ trivially.
Formula \eqref{eq:Direct-sum} then implies that 
$$\Delta(Q; B_1\oplus \dots \oplus  B_l)=\Delta(Q; B_1)\cup \dots\cup \Delta(Q;  B_l)$$
contains no lines. Moreover,
our construction shows that $\Delta(Q; B_1\oplus \dots \oplus  B_l)$ lies in the finite union of subspaces,
each either a member of $\Omega$ or one of the associated subspaces of the symplectic basis $\mathcal{C}$,
and generated by at most $k$ elements of $Q^{\#}$.

 We have completed the inductive step 
 and so we can now deduce that, if $A=B_1\oplus \dots \oplus  B_m $ 
 then $A$ is finitely generated as $\mathbb ZQ$-module and $\Delta(Q; A)$ contains no lines. 
 Further, if $\mathcal A$ is a finite generating set for $A$ then $Q'$ acts on $A$ in such a way that the split extension of $A$ by $Q'$ is locally polycyclic.  Thus $\Delta(Q'; \mathcal A\cdot \mathbb ZQ')$ is reduced to the origin and so certainly contains no lines. 
Thus, recalling Definition \ref{tame}, $A$ is tame. 

The group $G$ required by the theorem  is the split extension of $A$ by $Q$ and we deduce from Theorem \ref{RSthm} that it is finitely presented. The final step is to observe that, if $A$ were not the Fitting subgroup of $G$, then some non-trivial element of the centre of $Q$ would act nilpotently on $A$. 
The image of this element in some $Q_i$ would be non-trivial and still central and  would act nilpotently on $B_i$. 
But the construction of the action of the centre of each $Q_i$ makes it clear that, in all cases, 
no non-trivial element of the centre acts nilpotently on $B_i$. 
The proof is complete. 
\end{section}
\bibliography{Groves-Strebel_Fitting-quotients-rev}
\bibliographystyle{amsplain}
\end{document}